\documentclass[a4paper,11pt]{article}

\usepackage[margin=1in]{geometry}

\usepackage{graphicx,url,subfig}
\usepackage{multirow}
\usepackage{physics}
\usepackage{booktabs}
\RequirePackage[OT1]{fontenc}
\RequirePackage{amsthm,amsmath,amssymb,amscd}
\RequirePackage[numbers]{natbib}
\RequirePackage[colorlinks,citecolor=blue,urlcolor =blue]{hyperref}
\usepackage{enumerate}
\usepackage{datetime}
\makeatother
\numberwithin{equation}{section}
\allowdisplaybreaks
\usepackage{accents}

\usepackage{tikz}
\usetikzlibrary{
  arrows,
  arrows.meta,
  calc,
  decorations.fractals,
  decorations.pathreplacing,
  graphs,
  mindmap,
  positioning,
  shapes.geometric,
  through,shapes,patterns,
  trees,
  }

\usepackage{refcheck}

\theoremstyle{plain}
\newtheorem{theorem}{Theorem}[section]

\newtheorem{lemma}[theorem]{Lemma}
\newtheorem{proposition}[theorem]{Proposition}

\theoremstyle{definition}
\newtheorem{definition}[theorem]{Definition}

\newtheorem{remark}[theorem]{Remark}

\newtheorem{example}[theorem]{Example}

\newcommand{\W}{\dot{W}}
\newcommand{\ud}{\ensuremath{ \mathrm{d}} }
\newcommand{\Ceil}[1]{\left\lceil #1 \right\rceil}
\newcommand{\Floor}[1]{\left\lfloor #1 \right\rfloor}

\newcommand{\LFD}{I\!\!D}

\newcommand{\FoxH}[5]{H_{#2}^{#1}\left(#3\:\middle\vert\: \begin{subarray}{l}#4\\[0.4em] #5\end{subarray}\right)}

\newcommand{\calF}{\mathcal{F}}

\newcommand{\bbZ}{\mathbb{Z}}

\newcommand{\R}{\mathbb{R}}

\newcommand{\One}{\ensuremath{\mathrm{1}} }

\renewcommand{\Re}{\ensuremath{\mathrm{Re}} }

\newcommand{\DCap}[1]{\partial^{#1}}

\title{H\"older regularity of the nonlinear stochastic time-fractional\\ slow and fast diffusion equations on $\R^d$}

\author{
		Le Chen\\
		Emory University\\
		Department of Mathematics\\
		Atlanta, GA 30322\\
		\texttt{le.chen@emory.edu}\\
		\and
		Guannan Hu\\Washburn University\\
		Department of Mathematics  \\ and Statistics\\
		Topeka, KS 66606 \\
		\texttt{guannan.hu@washburn.edu}
		\date{\today}
}

\begin{document}
\maketitle

\begin{center}
\begin{minipage}[rct]{5 in}
\footnotesize

\textbf{Abstract:} In this paper, we use {\it local fraction derivative} to show the H\"older
continuity of the solution to the following nonlinear time-fractional slow and fast diffusion
equation:
\begin{align*}
		\left(\partial^\beta+\frac{\nu}{2}(-\Delta)^{\alpha/2}\right)u(t,x) = I_t^\gamma\left[\sigma\left(u(t,x)\right)\W(t,x)\right],\quad t>0,\: x\in\R^d,
\end{align*}
where $\W$ is the space-time white noise, $\alpha\in(0,2]$, $\beta\in(0,2)$, $\gamma\ge 0$ and
$\nu>0$, under the condition that $2(\beta+\gamma)-1-d\beta/\alpha>0$. The case when
$\beta+\gamma\le 1$ has been obtained in \cite{ChHuNu19}. In this paper, we have removed this extra
condition, which in particular includes all cases for $\beta\in(0,2)$.

\vspace{2ex}

\textbf{MSC 2010 subject classifications:} Primary 60H15. Secondary 60G60, 35R60.

\vspace{2ex}

\textbf{Keywords:} Fractional Taylor theorem, local fractional derivative, nonlinear stochastic time-fractional slow and fast
diffusion equations, H\"older continuity, \textit{Mittag-Leffler} function, \textit{Fox-H} function.
\vspace{4ex}
\end{minipage}
\end{center}

\section{Introduction}
In this article, we study the H\"older continuity of the solution of the following nonlinear
stochastic time-fractional slow and fast diffusion equations:
\begin{align}\label{E:SPDE}
\begin{cases}
		\left(\displaystyle \DCap{\beta} +\frac{\nu}{2}(-\Delta)^{\alpha/2} \right) u(t,x) =
		I_t^\gamma\left[\sigma\left(u(t,x)\right) \dot{W}(t,x) \right], & t>0,\:x\in\R^d,\\[0.2em]
		u(0,\cdot)=\mu & \text{if $\beta\in \:(0,1]$,}\\[0.2em]
		\displaystyle u(0,\cdot)=\mu_0,\quad\frac{\partial }{\partial t}u(0,\cdot)=\mu_1 & \text{if $\beta\in \:(1,2)\:$,}
\end{cases}
\end{align}
where $\nu>0$ is the diffusion parameter, the function $\sigma:\R\mapsto \R$ is globally Lipschitz
continuous, and $\W$ denotes the space-time white noise. In this equation, there are three
fractional operators: $(-\Delta)^{\alpha/2}$, $\alpha\in (0,2]$, refers to the standard fractional
Laplacian in the spatial variable; $\DCap{\beta}$ denotes the {\it Caputo fractional
derivative}
\begin{align}
		\label{E:caputuder}
		\DCap{\beta} f(t) :=
		\begin{cases}
				\displaystyle
				\frac{1}{\Gamma(m-\beta)} \int_0^t\ud \tau\: \frac{f^{(m)}(\tau)}{(t-\tau)^{\beta+1-m}}& \text{if $m-1<\beta<m$\;,}\\[1em]
				\displaystyle
				\frac{\ud^m}{\ud t^m}f(t)& \text{if $\beta=m$}\;,
		\end{cases}
\end{align}
and $I_{s+}^\gamma$ refers to the (left-sided) {\it Riemann-Liouville fractional integral} of order $\gamma> 0$:
\begin{align*}
		\left( I_{s+}^\gamma f \right)(t) := \frac{1}{\Gamma(\gamma)}\int_s^t (t-r)^{\gamma-1}f(r)\ud r,  \quad \text{for $t>s$},
\end{align*}
and the term $I_t^\gamma\left[\sigma(u(t,x))\W(t,x)\right]$ in \eqref{E:SPDE} refers to the more cumbersome notation
\begin{align*}
		\left(I_{0+}^\gamma \left[\sigma\left(u(\cdot,x)\right)\W(\cdot,x)\right]\right)(t).
\end{align*}
The existence and uniqueness of a random field solution has been established in \cite{ChHuNu19}
under the {\it Dalang's condition}: $d<\Theta$, which is equivalent to
\begin{align}
		\label{E:Dalang}
		\rho>0\quad\text{and}\quad d<2\alpha,
\end{align}
where the constants $\rho$ and $\Theta$ are defined as follows:
\begin{align}
		\label{E:RhoTheta}
		\rho&  := 2(\beta+\gamma)-1-d\beta/\alpha \qquad\text{and}\qquad
		\Theta := 2\alpha + \frac{\alpha}{\beta}\min(2\gamma-1,0).
\end{align}

The stochastic partial differential equations (SPDE's) with time-fractional differential operators
have received many attentions recently, which have been studied in one way or another using various
tools from stochastic analysis; see \cite{Chen14Time, CHHH15Time, ChHuNu19, ChenKimKim15, HuHu15,
LRD18, MijenaNane14ST}. This current paper follows the same setup as \cite{ChHuNu19}, where the
existence/uniqueness of the solution to \eqref{E:SPDE}, moment Lyapunov exponent, sample-path
H\"older regularity, etc., have been studied. In particular, the sample-path regularity was obtained
under the following additional condition:
\begin{align}
		\label{E:bg<1}
		\beta+\gamma\le 1,
\end{align}
which in particular excludes the case when $\beta\in (1,2)$. The question is whether one can get rid
of this additional requirement.

By carefully examining the corresponding proof in \cite{ChHuNu19}, one finds that the completely
monotonic property \cite{Schneider96CM} of the \textit{two-parameter Mittag-Leffler} function has
been applied in a crucial way. This complete monotonic property, applied to $E_{\beta,\beta+\gamma}
(-x)$, $x\ge 0$,  says that \begin{align} \label{E:CM} \begin{aligned} x\in[0,\infty) \mapsto
E_{\beta,\beta+\gamma} (-x) \:\:\text{is complete monotonic} \quad & \Longleftrightarrow \quad 0<
\beta \le \min(\beta+\gamma,1)\\ & \Longleftrightarrow \quad 0< \beta < 1 \:\:\text{and}\:\:
\gamma\ge 0, \end{aligned} \end{align} which in particular restricts the results in \cite{ChHuNu19}
only to the cases when $\beta\in (0,1]$.  Moreover, the proof in \cite{ChHuNu19} requires an even
stronger condition, namely, \eqref{E:bg<1}, than those on the far right side of \eqref{E:CM}.  It is
well known that when $\beta\in(1,2)$, the Mittag-Leffler function has finite (odd) number of
negative real zeros (see, e.g., Section 18.1 of \cite{ErdelyiEtc55HTF-V3}), or the function of
interest in \eqref{E:CM} is in general oscillatory when $\beta\in(1,2)$, which, together with the
fact that one needs to bound the Mittag-Leffler functions of the form in \eqref{E:CM} from below in
the proof of \cite{Chen14Time}, poses a significant challenge to extend the method used in
\cite{ChHuNu19} to cover the cases without the extra restriction \eqref{E:bg<1}. Instead, inspired
by the {\it local fractional derivatives} (LFD), we manage to obtain the desired regularity results
without condition \eqref{E:bg<1}, which covers the corresponding results obtained in \cite{ChHuNu19}
as a special case.  \bigskip

In order to state our main result, let us first recall that for a given subset $D\subseteq
[0,\infty)\times\R^d$ and positive constants $\beta_1$ and $\beta_2$, $C_{\beta_1,\beta_2}(D)$
denotes the set of {\it (locally) H\"older continuous functions over $D$ }, that is, any function
$v: [0,\infty)\times \R^d \mapsto \R$ with the property that for each compact set $K\subseteq D$,
there is a finite constant $C$ such that for all $(t,x)$ and $(s,y)\in K$,
\begin{align*}
		\vert v(t,x) - v(s,y) \vert \leq C \left(\vert t-s\vert^{\beta_1} + \vert x-y \vert^{\beta_2}\right).
\end{align*}
Denote $C_{\beta_1-,\beta_2-}(D) := \cap_{\alpha_1\in \;(0,\beta_1)}
\cap_{\alpha_2\in \;(0,\beta_2)} C_{\alpha_1,\alpha_2}(D)$.

\begin{theorem}
		\label{T:main}
		Let $u(t,x)$ be the solution to \eqref{E:SPDE} under Dalang's condition \eqref{E:Dalang} and assume
		that the initial conditions $\mu$, $\mu_0$ and $\mu_1$ are such that
		\begin{align}
				\label{E:InitCt}
				\sup_{(s,x)\in[0,t]\times\R^d} |J_0(s,x)|<\infty,\qquad\text{for all $t\ge 0$},
		\end{align}
		where $J_0(t,x)$ refers to the solution to the homogeneous equation of \eqref{E:SPDE} (see
		\eqref{E:J0}). Then
		\begin{align}
				\label{E:Holder-u}
				u(\cdot,\cdot) \in C_{\frac{1}{2}\min(\rho, 2)-, \:\frac{1}{2}\min(\Theta-d,2)-}\left((0,\infty)\times\R^d\right),\quad \text{a.s.}.
		\end{align}
\end{theorem}

\begin{remark}
		\label{R:Old}
    Under the extra condition $\beta+\gamma\le 1$, we see that $\rho\le 1-d\beta/\alpha<1$ and
    hence, the time exponent $\frac{1}{2}\min(\rho,2)$ reduces to $\rho/2$, which recovers the
    result in \cite{ChHuNu19}; see (1.17) or Theorem 3.3, {\it ibid}.  When $\beta=1$, $\gamma=0$
    and $d=1$, we have a stochastic heat equation with a fractional Laplacian. In this case, the
    exponents become $\left(\left[\frac{1}{2}-\frac{1}{2\alpha}\right]-,\frac{\alpha-1}{2}-\right)$;
    see, e.g., \cite[Proposition 4.4]{ChenDalang15FracHeat} for a reference and Figure \ref{F:1d}.a
    for an illustration.
\end{remark}

\begin{figure}[htb]
	\centering
  \begin{tikzpicture}[scale=2.3]
    \draw[->] (-0.2,0) -- (2.25,0) node[right]{$\alpha$};
    \draw[->] (0,-0.6) -- (0,1.25);
    \draw[very thick] (1,0) -- (2,1/2) node[midway,above,sloped]{space};
    \draw[very thin,dashed] (0,-0.5) -- (1,0);
    \def\inc{0.04}
    \draw (\inc,-0.5) -- (-\inc,-0.5) node[left]{$-1/2$};
    \draw (\inc,0.5) -- (-\inc,0.5) node[left]{$1/2$};
    \draw (\inc,0.25) -- (-\inc,0.25) node[left]{$1/4$};
    \draw (\inc,1) -- (-\inc,1) node[left]{$1$};
    \draw (1,\inc) -- (1,-\inc) node[below]{$1$};
    \draw (1,\inc) -- (1,-\inc) node[below]{$1$};
    \draw (2,\inc) -- (2,-\inc) node[below]{$2$};
    \draw (0.667,\inc) -- (0.667,-\inc) node[below]{$2/3$};
    \draw[domain=1:2, smooth, variable=\x, very thick,dotted] plot ({\x}, {0.5 - 1/(2*\x)});
    \draw[very thin,dashed] (0, 0.25) -- (2,0.25) -- (2,0);
    \draw[very thin,dashed] (0, 0.5) -- (2,0.5) -- (2,0);
    \filldraw (2,0.25) circle (0.03);
    \filldraw (2,0.5) circle  (0.03);
    \draw[thick] (1,0) circle (\inc);
    \node[rotate=11] (a) at (1.7,0.1) {time};
    \node (a) at (1,-0.7) {(a) $\beta=1$};
  \end{tikzpicture}
  \begin{tikzpicture}[scale=2.3]
    \draw[->] (-0.2,0) -- (2.25,0) node[right]{$\beta$};
    \draw[->] (0,-0.6) -- (0,1.25);
    \draw[very thick,dotted] (0.667,0) -- (2,1) node[midway,below,sloped]{time};
    \draw[very thin, dashed] (0,-0.5) -- (0.667,0);
    \def\inc{0.04}
    \draw (\inc,-0.5) -- (-\inc,-0.5) node[left]{$-1/2$};
    \draw (\inc,0.5) -- (-\inc,0.5) node[left]{$1/2$};
    \draw (\inc,0.25) -- (-\inc,0.25) node[left]{$1/4$};
    \draw (\inc,1) -- (-\inc,1) node[left]{$1$};
    \draw (1,\inc) -- (1,-\inc) node[below]{$1$};
    \draw (1,\inc) -- (1,-\inc) node[below]{$1$};
    \draw (2,\inc) -- (2,-\inc) node[below]{$2$};
    \draw (0.667,\inc) -- (0.667,-\inc) node[below]{$2/3$};
    \draw[domain=0.667:2, smooth, variable=\x, very thick] plot ({\x}, {1.5 - 1/\x});
    \draw[very thin,dashed] (0, 1) -- (2,1) -- (2,0);
    \draw[very thin,dashed] (0, 0.25) -- (1,0.25) -- (1,0);
    \draw[very thin,dashed] (0, 0.5) -- (1,0.5) -- (1,0);
    \filldraw (1,0.25) circle (0.03);
    \filldraw (1,0.5) circle  (0.03);
    \draw[thick] (2,1) circle (\inc);
    \draw[thick] (0.667,0) circle (\inc);
    \node[rotate=35] (a) at (1.1,0.73) {space};
    \node (a) at (1,-0.7) {(b) $\alpha=2$};
  \end{tikzpicture}
  \caption{The exponents for the case $\gamma=0$ and $d=1$.}
	\label{F:1d}
\end{figure}
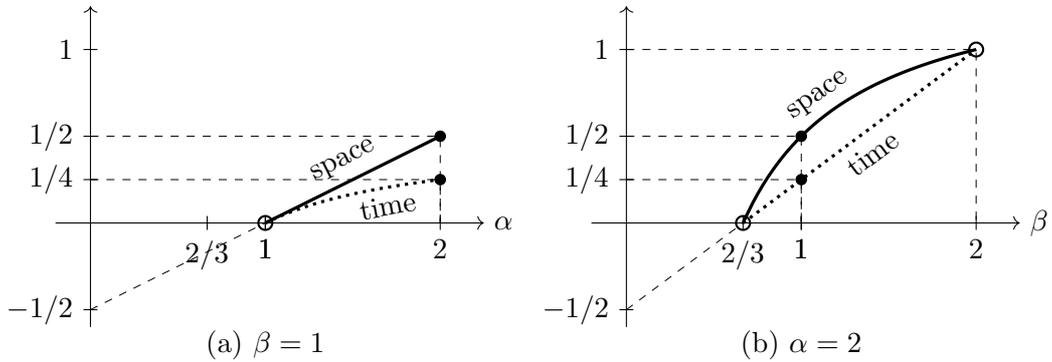

\begin{remark}
  \label{R:Compare}
  When $\gamma=0$, $\alpha=2$ and $d=1$, the results in  \eqref{E:Holder-u}	reduces to
  \begin{align*}
    u(\cdot,\cdot) \in C_{\left(\frac{3}{4}\beta-\frac{1}{2}\right)-, \:\left(\frac{3}{2}-\frac{1}{\beta}\right)-}\left((0,\infty)\times\R^d\right),\quad \text{a.s.};
  \end{align*}
  See Figure \ref{F:1d}.b for the time and space exponents. It is clear that when $\beta=1$, our
  results recovers the standard exponents $(\frac{1}{4}-,\frac{1}{2}-)$ for the stochastic heat
  equation. However, when $\beta=2$, it is known that the stochastic wave equation has the exponents
  $(\frac{1}{2}-,\frac{1}{2}-)$ ; see, e.g., Part (6) of Theorem 4.2.1 in \cite{Chen13thesis}.  On the
  other hand, our results require $\beta$ to be strictly smaller than $2$. But when $\beta$ is very
  close to $2$, our results suggest the exponents to be $(1-,1-)$, which is rather surprising.  This
  phase transition of the H\"older regularity from $\beta<2$ to $\beta=2$ has also been reflected by the
  degeneration of the fundamental solution $x\ne 0\mapsto Y(t,x)$ (see Section \ref{S:proof-T1} for
  the notation) with $t>0$ fixed from a smooth function to a piece-wise continuous function.
\end{remark}

The proof of this theorem relies on a crucial estimate given in Proposition \ref{mP:G-SD} below. In
Section \ref{S:proof-T1}, we introduce some notation and preliminaries including Proposition
\ref{mP:G-SD}, based on which we streamline the proof of Theorem \ref{T:main}. Then we prepare some
auxiliary results in Sections \ref{S:Cont} and \ref{S:Holder}, before we finally prove Proposition
\ref{mP:G-SD} in Section \ref{S:Prop}.

\section{Notation, Preliminaries, and Proof of Theorem \ref{T:main}}
\label{S:proof-T1}

We use $\Ceil{x}$ and $\Floor{x}$ to denote the conventional ceiling and floor functions, respectively.
Recall that (see \cite{ChHuNu19}) the fundamental solutions to \eqref{E:SPDE} consist of the
following three functions:
\begin{gather}
		Z_{\alpha,\beta,d}(t,x)        := \pi^{-d/2} t^{\Ceil{\beta}-1} |x|^{-d} \FoxH{2,1}{2,3}{\frac{ |x|^\alpha}{2^{\alpha-1}\nu t^\beta}}{(1,1),\:(\Ceil{\beta},\beta)} {(d/2,\alpha/2),\:(1,1),\:(1,\alpha/2)},\\
		Z^*_{\alpha,\beta,d}(t,x)      := \pi^{-d/2} |x|^{-d} \FoxH{2,1}{2,3}{\frac{ |x|^\alpha}{2^{\alpha-1}\nu t^\beta}}{(1,1),\:(1,\beta)} {(d/2,\alpha/2),\:(1,1),\:(1,\alpha/2)},\\ \label{E:Yab}
		Y_{\alpha,\beta,\gamma,d}(t,x) := \pi^{-d/2} |x|^{-d}t^{\beta+\gamma-1} \FoxH{2,1}{2,3}{\frac{ |x|^\alpha}{2^{\alpha-1}\nu t^\beta}} {(1,1),\:(\beta+\gamma,\beta)}{(d/2,\alpha/2),\:(1,1),\:(1,\alpha/2)},
\end{gather}
where $\FoxH{2,1}{2,3}{*}{\cdot}{-}$ denotes the \textit{Fox $H$-function} (see, e.g.,
\cite{KilbasSaigo04H}). For simplicity, in the following we will simply write these functions as
$Z$, $Z^*$, and $Y$.  The solution to \eqref{E:SPDE} is understood as the {\it mild solution}
\begin{align*}
		u(t,x) = J_0(t,x) + \iint_{[0,t]\times\R^d} Y(t-s,x-y)\sigma(u(s,y)) W(\ud s, \ud y),
\end{align*}
with the stochastic integral being the \textit{Walsh integral} \cite{Walsh86}, and $J_0(t,x)$ is the
solution to the homogeneous equation, namely,
\begin{align}\label{E:J0}
		J_0(t,x)=
		\begin{cases}
				\left(Z(t,\cdot)*\mu\right)(x)                                      & \text{if $\beta\in(0,1]$,}\\[0.2em]
				\left(Z^*(t,\cdot)*\mu_0\right)(x)+\left(Z(t,\cdot)*\mu_1\right)(x) & \text{if $\beta\in(1,2)$.}\\[0.2em]
		\end{cases}
\end{align}
SPDE's of the above form have been widely studied; see, e.g.,
\cite{ChenDalang15Heat,Dalang99Extending,FoondunKhoshnevisan08Intermittence} and references
therein.
\bigskip

In \cite{ChHuNu19}, Theorem \ref{T:main} was proved in the case of $\beta+ \gamma \le 1$ via the
following continuity results on $Y(t,x)$:

\begin{proposition}[Proposition 5.4 of \cite{ChHuNu19}]
		\label{P:G-SD}
		Suppose $\alpha \in (0,2]$, $\beta\in(0,2)$ and $\gamma\ge 0$. Under Dalang's condition
		\eqref{E:Dalang}, $Y(t,x)$ satisfies the following two properties:

		\noindent (i) For all $0<\theta<(\Theta-d)\wedge 2$ and $T>0$, there is some nonnegative
		constant $C=C(\alpha,\beta,\gamma,\nu,\theta,T,d)$ such that for all $t\in (0,T]$ and
		$x,y\in\R^d$,
		\begin{align}\label{E:G-x}
				\iint_{\R_+\times\R^d}\ud r\ud z\: \left( Y\left(t-r,x-z\right) -Y\left(t-r,y-z\right) \right)^2\le C \: |x-y|^{\theta}.
		\end{align}

		\noindent (ii) If $\beta\le 1$ and $\beta + \gamma \le 1$, then there is some nonnegative
		constant $C=C(\alpha,\beta,\gamma,\nu,d)$ such that for all $s,t\in(0,\infty)$ with $s\le t$,
		and $x\in\R^d$,
		\begin{gather}
				\label{E:G-t1}
				\int_0^s \ud r\int_{\R^d}\ud z \left(Y\left(t-r,x-z\right) -Y\left(s-r,x-z\right) \right)^2\le C (t-s)^{2(\beta+\gamma)-1-d\beta/\alpha}\;,\\
				\label{E:G-t2}
				\int_s^t\ud r\int_{\R^d}\ud z \; Y^2\left(t-r,x-z\right) \le C (t-s)^{2(\beta+\gamma)-1-d\beta/\alpha}\;.
		\end{gather}
\end{proposition}

The constraint $\beta+\gamma\le 1$ comes from part (ii). Hence, in order to prove our main result,
we need only to establish a slightly weaker version of $\eqref{E:G-t1}$ as stated in the next
proposition:

\begin{proposition}\label{mP:G-SD}
		Let $\alpha\in (0,2]$, $\beta\in(0,2)$, $\gamma\ge 0$, $0\leq s\le t \leq T<\infty$, $x\in\R^d$,
		and $\rho$ be the constant defined in \eqref{E:RhoTheta}. Assume that Dalang's condition
		\eqref{E:Dalang} holds. If $\rho <2$, then there exists a nonnegative constant
		$C=C(\alpha,\beta,\gamma,\nu,d,T)$ such that for all $q\in(0,\rho)$, it holds that
		\begin{align}
				\label{mE:G-t2}
				\int_0^s \ud r\int_{\R^d}\ud z \left(Y\left(t-r,x-z\right) -Y\left(s-r,x-z\right) \right)^2\le C (t-s)^{q}\;.
		\end{align}
		If $\rho>2$, \eqref{mE:G-t2} is true with $q$ replaced by $2$.
\end{proposition}

This proposition is proved in Section \ref{S:Prop}. Now we are ready to prove our main result.

\begin{proof}[Proof of Theorem \ref{T:main}]
		The proof follows the same arguments as those in the proof of Theorem 3.3 in \cite{ChHuNu19},
		Theorem \ref{T:main} is immediate via equations \eqref{E:G-x} and \eqref{E:G-t2} in Proposition
		\ref{P:G-SD} and Proposition \ref{mP:G-SD}. See also the proof of Proposition 4.3 in
		\cite{ChenDalang14Holder} for a more detailed exposition of these arguments.
\end{proof}
\bigskip

Now let us introduce some more notation for this paper. For any $\gamma_1,\gamma_2\in\R$, denote
\begin{align}
		\label{E:Csharp}
	\begin{aligned}
		C_{\gamma_1,\gamma_2} &:= \frac{2^{1-d+d/ \alpha}}{\Gamma(d/2)\pi^{d/2}\nu^{d/ \alpha}} \int_0^\infty u^{d-1}E_{\beta,\beta+\gamma_1}(-u^\alpha)E_{\beta,\beta+\gamma_2}(-u^\alpha)\ud u \quad \text{and}\quad C_{\gamma}:=C_{\gamma,\gamma},\\
		C_{\gamma_1,\gamma_2}^* &:= \frac{2^{1-d+d/ \alpha}}{\Gamma(d/2)\pi^{d/2}\nu^{d/ \alpha}} \int_0^\infty u^{d-1}\left| E_{\beta,\beta+\gamma_1}(-u^\alpha)E_{\beta,\beta+\gamma_2}(-u^\alpha)\right| \ud u \quad \text{and}\quad C_{\gamma}^*:=C_{\gamma,\gamma}^*,
	\end{aligned}
\end{align}
where $E_{\beta,\beta+\gamma}(\cdot)$ is the two-parameter Mittag-Leffler function and these
constants are well defined provided that $d<2\alpha$ (see Dalang's condition \eqref{E:Dalang})
thanks to the following asymptote property (see, e.g., Theorem 1.6 of \cite{Podlubny99FDE}):
\begin{align}\label{E:AsypMTL}
		|E_{\beta,\zeta}(z)| \leq \frac{A}{1+|z|} \qquad \text{for $\beta \in (0, 2),\: \zeta \in \R,\:  |\arg(z)|=\pi$.}
\end{align}
It is clear that $C_{\gamma_1,\gamma_2}^*>0$ and $C_{\gamma}^* = C_{\gamma}>0$,
but $C_{\gamma_1,\gamma_2}$ may take negative values.  As shown in \cite[Lemma 5.5]{ChHuNu19}, using
the Fourier transform of $Y_{\alpha,\beta,\gamma,d}(t,\cdot)$, namely,
\begin{align}\label{E:FY}
		\calF Y_{\alpha,\beta,\gamma,d}(t,\cdot)(\xi)& = t^{\beta+\gamma-1} E_{\beta,\beta+\gamma}(-2^{-1}\nu t^\beta |\xi|^\alpha),
\end{align}
and the Plancherel theorem, we have
\begin{align}
		\label{E:int-Y^2}
		\int_{\R^d} Y_{\alpha,\beta,\gamma,d}^2(t,x) \ud x = C_{\gamma} \: t^{2(\beta+\gamma-1)-d\beta/\alpha},
		\quad \text{for all $t>0$.}
\end{align}
Following the proof of part (ii) of Proposition 5.4 in \cite{ChHuNu19} or by \eqref{E:int-Y^2}, we have that
\begin{align*}
		\int_0^s \ud r\int_{\R^d}\ud z \left[ Y\left(t-r,x-z\right) -Y\left(s-r,x-z\right)\right ]^2
		= \frac{ C_{\gamma}}{\rho}\left[t^{\rho}-(t-s)^{\rho} + s^{\rho}\right]-2h_s(t),
\end{align*}
where
\begin{align}
		\label{E:defht}
		\begin{aligned}
				h_s(t):= \frac{1}{(2\pi)^d}\int_0^s\ud r\int_{\R^d}\ud \xi
				\quad  & (t-r)^{\beta+\gamma-1} E_{\beta,\beta+\gamma}\left(-2^{-1}\nu(t-r)^\beta |\xi|^\alpha \right) \\
				\times & (s-r)^{\beta+\gamma-1}E_{\beta,\beta+\gamma}\left(-2^{-1}\nu(s-r)^\beta |\xi|^\alpha\right).
		\end{aligned}
\end{align}
For all $t,s\ge r$, $\xi\in\R^d$ and $\delta\in\R$, denote
\begin{align}
		\label{E:frackE}
		\begin{aligned}
				\mathfrak{E}(t,\delta,r,\xi)
				:= & \frac{1}{(2\pi)^d}(t-r)^{\delta-1} E_{\beta,\delta}\left(-2^{-1}\nu(t-r)^\beta |\xi|^\alpha \right) \\
           & \times (s-r)^{\beta+\gamma-1}E_{\beta,\beta+\gamma}\left(-2^{-1}\nu(s-r)^\beta |\xi|^\alpha\right).
		\end{aligned}
\end{align}
To obtain the estimation in Proposition \ref{mP:G-SD}, it boils down to the H\"older continuity of
$[s,T]\ni t\mapsto h_s(t)$ at $t=s$. Note that it is more convenient to consider $h_s(t)$ as a
function of $t$ with $s$ fixed than the other way round. We will investigate H\"older continuity of
$h_s(t)$ at $t=s$ via its \textit{local fractional derivative} (LFD) in Section \ref{S:Holder}, which
will be used to prove Proposition \ref{mP:G-SD} in Section \ref{S:Prop}.
\bigskip

Finally, we will need the following two lemmas:

\begin{lemma}[Formula 23 of Table 9.1 on p. 173 of \cite{SamkoKilbasMarichev93}]
		\label{L:fracdi}
		For $\Re(\beta)>0$, $\Re(\alpha)>0$, $q\in\R$, and $\lambda>0$, it holds that
		\begin{align}
				\label{E:fracdi}
				\left( I_{r+}^{q}\left[ (\cdot-r)^{\beta-1} E_{\alpha, \beta}\left( \lambda (\cdot-r)^{\alpha}\right)\right ]\right)(t) = (t-r)^{\beta+q-1} E_{\alpha, \beta+q}\left(\lambda (t-r)^{\alpha}\right)
				\quad \text{for $t>r$},
		\end{align}
		where we use the convention that $I_{r+}^{q}= D_{r+}^{-q}$ for $q<0$; see \eqref{E:RL-Dev-Def}
		below for the notation.
\end{lemma}
\begin{remark}
		\label{R:fracdi}
		When $q$ is a negative integer, $I_{r+}^{q}$ becomes the conventional derivative, which ensures
		that \eqref{E:fracdi} also holds for $\beta<0$.
\end{remark}

\begin{lemma}
		\label{L:Y^2}
		Let $1\le d<2\alpha$, $\beta\in (0,2)$, $\delta,\delta'\in\R$. Then for
		all $t,s>0$,
		\begin{align*}
				\int _{\R^d}\left| E_{\beta ,\beta+\delta }\left( -\dfrac {1}{2}\nu t^{\beta }|x|^{\alpha } \right) \times
				E_{\beta,\beta+\delta'}\left(-\dfrac {1}{2}\nu s^{\beta}|x|^{\alpha} \right)\right| \ud x\leq
				C_{\delta,\delta'}\left[ t\cdot s\right ] ^{-\frac {\beta}{2\alpha }d}.
		\end{align*}
\end{lemma}

The proof of this lemma follows immediately from the arguments of \cite[Lemma 5.5]{ChHuNu19} with an
application of the Cauchy-Schwarz inequality; note that the condition $\beta\in(0,2)$ is required
due to \eqref{E:AsypMTL}.

\section{Continuity of \texorpdfstring{$h_s^{(n)}(t)$}{}}
\label{S:Cont}

Throughout of the rest of the article, unless specified otherwise, let $q$ and
$n$ be given by
\begin{align}
		\label{E:nq}
		0\leq n<q<\rho\leq n+1 \qquad\text{ with $n\in \mathbb{Z}_+$.}
\end{align}
The aim of this section is to show that $h_s^{(n)}(t)$, defined in \eqref{E:defht}, is continuous
for $t\in[s,T]$ and compute $h_s'(t)$ and $h_s''(t)$ at $t=s$, in case of $\rho>1$ and $\rho>2$,
respectively.

\begin{lemma}\label{L:intdevh(s)}
		Assuming \eqref{E:nq}, for any $0<s<T<\infty$, it holds that $h_s^{(n)}(\cdot)\in C([s,T])$. Moreover,
		\begin{align*}
				h_s^{(n)}(t)\Big|_{t=s} =
				\begin{cases}
						C_{\gamma} \: \frac12 s^{\rho-1}                                         & \text{with $n=1$ when $\rho>1$}, \\[1em]
						\left[ C_{\gamma,\gamma-1}-\frac{C_{\gamma-1}}{\rho-2} \right]s^{\rho-2} & \text{with $n=2$ when $\rho>2$}, \\
				\end{cases}
		\end{align*}
		where $C_{\gamma}$ and $C_{\gamma_1,\gamma_2}$ are defined in \eqref{E:Csharp}.
\end{lemma}

\begin{proof}
		In the proof, we will use $C$ to denote a generic constant that may change its value at each
		appearance. When $ \rho>n\geq 0$, we first claim that one can switch the double integrals and
		differentiation in order to have that
		\begin{align}
				\label{E:h^(n)}
				h_s^{(n)}(t)=\int_0^s \ud r\int_{\R^d} \ud \xi\:\: &\dv{^n}{t^n} \: \mathfrak{E}(t,\beta+\gamma,r,\xi), \quad \text{for all $t\in[s,T]$.}
		\end{align}
		Indeed, apply Lemma \ref{L:fracdi} with $q=-n$ (see Remark \ref{R:fracdi}) to see that,
		\begin{align*}
				\left|\: \dv{^n}{t^n} \: \mathfrak{E}(t,\beta+\gamma,r,\xi)\right|
				& = \left|\mathfrak{E}(t,\beta+\gamma-n,r,\xi)\right| \\
				& \le C \frac{1}{1+2^{-1}\nu (t-r)^\beta |\xi|^\alpha} \times \frac{1}{1+2^{-1}\nu (s-r)^\beta |\xi|^\alpha} \\
				& \le C \left( 1+2^{-1}\nu (s-r)^\beta |\xi|^\alpha \right)^{-2},
		\end{align*}
		where we have applied the asymptotic property \eqref{E:AsypMTL}. Notice that the upper bound
		does not depend on $t$ and is $\ud \xi$ integrable:
		\begin{align*}
				\int_{\R^d}	\left( 1+2^{-1}\nu (s-r)^\beta |\xi|^\alpha \right)^{-2} \ud\xi
				= C (s-r)^{-\frac{d\beta}{\alpha}}\int_0^\infty \frac{u^{d-1}}{(1+u^\alpha)^2}\ud u,
		\end{align*}
		which is finite because $d<2\alpha$. Therefore, one can apply Fubini's theorem to see that
		\begin{align}
				\label{E_:Fubini_1}
				\int_{\R^d} \ud \xi\:\: \dv{^n}{t^n} \: \mathfrak{E}(t,\beta+\gamma,r,\xi)
				= \dv{^n}{t^n} \int_{\R^d} \ud \xi\: \mathfrak{E}(t,\beta+\gamma,r,\xi)
				, \quad \text{for all $t\in[s,T]$.}
		\end{align}

    Moreover, notice that, by the Cauchy-Schwarz inequality for the $\ud \xi$-integral,
		\begin{align*}
				\left| \dv{^n}{t^n} \: \int_{\R^d}\ud\xi \: \mathfrak{E}(t,\beta+\gamma,r,\xi)\right|
        & \le \int_{\R^d}\ud \xi \left| \mathfrak{E}(t,\beta+\gamma-n,r,\xi) \right|                                                                   \\
        & \le \frac{1}{(2\pi)^d} (t-r)^{\beta+\gamma-n-1}(s-r)^{\beta+\gamma-1}                                                                      \\
        & \quad \times \left(\int_{\R^d}\ud \xi \:\left| E_{\beta,\beta+\gamma-n}\left(-2^{-1}\nu(t-r)^\beta|\xi|^\alpha\right)\right| \right)^{1/2} \\
        & \quad \times \left(\int_{\R^d}\ud \xi \:\left| E_{\beta,\beta+\gamma}\left(-2^{-1}\nu(s-r)^\beta|\xi|^\alpha\right)\right|\right)^{1/2}    \\
        & = C_* (t-r)^{\beta+\gamma-n-1-\frac{d\beta}{2\alpha}}(s-r)^{\beta+\gamma-1-\frac{d\beta}{2\alpha}},
		\end{align*}
		where the last equality is obtained through change of variables followed by applications of
		\eqref{E:Csharp}, and $C_*:=\sqrt{C_{\gamma-n} C_{\gamma}}$. We claim that
		\begin{align}
				\label{E:neg-coef}
			  \beta+\gamma-n-1-\frac{d\beta}{2\alpha}\le 0.
		\end{align}
		Otherwise, we will have $\beta+\gamma-1-\frac{d\beta}{2\alpha}> 0$ as well, which implies that
		\begin{align*}
				0<
				\left(\beta+\gamma-1-\frac{d\beta}{2\alpha}\right) +
			  \left( \beta+\gamma-n-1-\frac{d\beta}{2\alpha} \right)
				= \rho - n -1,
		\end{align*}
		which contradicts \eqref{E:nq}. Thanks to \eqref{E:neg-coef}, we can replace $t$ by $s$ to see
		that
		\begin{align*}
				\left| \dv{^n}{t^n} \: \int_{\R^d}\ud\xi \: \mathfrak{E}(t,\beta+\gamma,r,\xi)\right|
				\le C_* (s-r)^{2(\beta+\gamma)-n-2-\frac{d\beta}{\alpha}}
				= C_* (s-r)^{\rho-n-1},
		\end{align*}
		with the right-hand side both independent of $t$ and integrable with respect to $\ud r$:
		\begin{align*}
				\int_0^s C_* (s-r)^{\rho-n-1}\ud r= \frac{C_*}{\rho-n}s^{\rho-n},
		\end{align*}
		which is finite because $\rho>n$.
		Therefore, one can apply Fubini's theorem again to see that, for all $t\in[s,T]$,
		\begin{align}
				\label{E_:Fubini_2}
				\int_0^s\ud r\: \dv{^n}{t^n} \: \int_{\R^d}\ud\xi \: \mathfrak{E}(t,\beta+\gamma,r,\xi)
				=  \dv{^n}{t^n} \: \int_0^s\ud r \int_{\R^d}\ud\xi \: \mathfrak{E}(t,\beta+\gamma,r,\xi).
		\end{align}
		A combination of \eqref{E_:Fubini_1} and \eqref{E_:Fubini_2} proves \eqref{E:h^(n)}. Moreover,
		the bounds obtained from the above two Cases I and II prove the following useful bound
		\begin{align}
				\label{E:intbdd_h^(n)}
				\int_0^s \ud r\int_{\R^d}\ud\xi \:\left|\: \dv{^n}{t^n} \: \mathfrak{E}(t,\beta+\gamma,r,\xi)\right|
				\le \frac{C_*}{\rho-n}s^{\rho-n}<\infty.
		\end{align}
		Another direct consequence of \eqref{E:h^(n)} is that $t\mapsto h^{(n)}_s(t)$ is a continuous
		function on $[s,T]$.  \bigskip

		Now we are ready to calculate $h_s'(t)|_{t=s}$ and $h_s''(t)|_{t=s}$ by passing the derivatives all
		the way inside the double integrals. In case of $ \rho>1$, we apply Lemma \ref{L:fracdi} to \eqref{E:h^(n)}
		for $n=1$ to see that
		\begin{align*}
			h_s'(t)\Big|_{t=s}=\int_{\R^d}\ud \xi\int^s_0\ud r\: & \mathfrak{E}(t,\beta+\gamma-1,r,\xi).
		\end{align*}
		Thanks to Lemma \ref{L:fracdi}, one can integrate $\ud r$ to see that
		\begin{align*}
				h_s'(t)\big|_{t=s}
				= \frac{1}{(2\pi)^d}\int_{\R^d} \frac12 \left[ s^{\beta+\gamma-1} E_{\beta,\beta+\gamma}\left(-2^{-1}\nu(s)^\beta |\xi|^\alpha \right)\right]^2\ud \xi
				= \frac{C_{\gamma}}{2} s^{2(\beta+\gamma-1)-\frac{d\beta}{\alpha}}=\frac{C_{\gamma}}{2} s^{\rho-1}.
		\end{align*}
		If $\rho>2$, apply Lemma \ref{L:fracdi} to  \eqref{E:h^(n)} for $n=2$ and integrate $h''(t)$ by parts, we have
		\begin{align*}
				h_s''(t)\Big|_{t=s} & = \int_{\R^d}\ud \xi \: \mathfrak{E}\left(s,\beta+\gamma-1,r,\xi\right)\Big|^{r=0}_{r=s}                                                                                     \\
                            & \quad-\frac{1}{(2\pi)^d}\int_{\R^d}\ud \xi\:\left[(s-r)^{\beta+\gamma-2} E_{\beta,\beta+\gamma-1}\left(-2^{-1}\nu(s-r)^\beta |\xi|^\alpha \right)\right]^2\Bigg|^{r=0}_{r=s} \\
                            & = C_{\gamma,\gamma-1} s^{2(\beta+\gamma)-3-d\beta/\alpha}-C_{\gamma-1}\frac{1}{\rho-2}s^{2(\beta+\gamma)-3-d\beta/\alpha}                                                    \\
                            & = \left[C_{\gamma,\gamma-1}-\frac{C_{\gamma-1}}{\rho-2} \right]s^{\rho-2},
		\end{align*}
		where both constants $C_{\gamma-1}$ and $C_{\gamma,\gamma-1}$ are defined in \eqref{E:Csharp}.
\end{proof}

\section {H\"older continuity of \texorpdfstring{$t\mapsto h_s^{(n)}(t)$}{} via local fractional derivative}
\label{S:Holder}

In this section, we will calculate the \textit{local  fractional  derivative} (LFD) of $h_s(t)$ at
$t=s$, and then prove the following Proposition \ref{P:Holder_h} about the H\"older continuity of
$h_s(t)$ at $t=s$:

\begin{proposition}\label{P:Holder_h}
		For all $t\in[s, \infty)$, for $n$ and $q$ defined in \eqref{E:nq}, it holds that
		\begin{align*}
				h_s(t)-\sum_{k=0}^{n}\frac{h_s^{(k)}(s)}{k!}(t-s)^k= o\left[(t-s)^{q}\right].
		\end{align*}
\end{proposition}

\begin{proof}
		The proof of this proposition is an application of the Fractional Taylor theorem (see Lemma
\ref{L:ftt}) to $h_s(t)$ with all conditions verified in Lemma \ref{L:caputo_h} below.
\end{proof}

In the following, we first give a brief overview of about LFD in Section \ref{S:LFD-intro} below. We
refer the interested reader to \cite{YBS16, KoLFDrev13, KoGaLFD} for more details about LFD.

\subsection{Local fractional derivative}
\label{S:LFD-intro}

Recall that for $-\infty<s<T<\infty$ and $q\ge 0$, the (left-sided) \textit{Riemann-Liouville fractional  integral} of
order $q$ of $f\in L^{1}([s,T];\R)$ is defined by (see, e.g., \cite[(2.17), p.33]{SamkoKilbasMarichev93})
\begin{align}
    \label{E:RL-Int-Def}
		\left( I_{s+}^{q} f \right)(t):= \frac{1}{\Gamma(q)}\int_s^t f(y)(t-y)^{q-1}\ud y, \quad t\in (s,T),
\end{align}
and the corresponding (left-sided) fractional derivative of order $q$ is defined formally by (see, e.g.,
\cite[(2.22), p.35]{SamkoKilbasMarichev93})
\begin{align}
    \label{E:RL-Dev-Def}
		\left( D_{s+}^q f \right)(t):=\left(\dv{ }{t}\right)^{n+1} \left( I_{s+}^{n+1-q} f \right)(t)={1\over \Gamma(n+1-q)}\left(\dv{ }{t}\right)^{n+1}\int_s^t f(y)(t-y)^{n-q}\ud y.
\end{align}
When $q=n$, the above definition gives $f^{(n)}(t)$ (if $f^{(n)}(t)$ exists) on $ (s, T)$.

\begin{definition}
		For arbitrary $q>0$, the (left-sided) \textit{local fractional derivative}
		(LFD) of $f(t)$ of order $q$ at $t=s$ is defined as, provided the limit exits,
		\begin{align}
				\label{E:LFD-Def}
				\LFD^qf(s) := \lim_{t\rightarrow s+}  \left( D_{s+}^{q} \left[f(\cdot)-\sum_{k=0}^{n}\frac{f^{(k)}(s)}{k!}(\cdot-s)^k\right] \right)(t),\quad 0<n<q\le n+1.
		\end{align}
\end{definition}
Clearly, when $q \in \mathbb{Z}_+$, $\LFD^qf(s) = f^{(q)}(s)$. Kolwankar and Gangal
\cite{KoGaLFD,KoGaLFDfp} introduced the above LFD of order $q\in(0,1)$ to investigate the regularity
of non-differentiable functions.  
Recall that the \textit{(left-sided) Caputo fractional derivative} is defined as (see, e.g., \cite[(2.4.1), p.91]{KilbasEtc06})
\begin{align}
		\label{E:CapDef}
		\left( \partial_{s+}^q f \right)(t) := \left( D_{s+}^{q} \left[f(\cdot)-\sum_{k=0}^{n}\frac{f^{(k)}(s)}{k!}(\cdot-s)^k\right] \right)(t),\qquad n = \Ceil{q}-1,
\end{align}
whereas the notation $\partial^\beta$ in \eqref{E:caputuder} is a short-hand notation for $\partial_{0+}^\beta$.

Comparing \eqref{E:LFD-Def} and \eqref{E:CapDef}, one
sees that the LFD is nothing but the limit of Caputo fractional derivative, namely, $\LFD^qf(s)
=\lim_{t\rightarrow s} \left( \partial_{s+}^q f \right)(t)$. When $f^{(n)}(t)$ is absolutely continuous on $[s,T]$,
both the Riemann-Liouville derivative and the Caputo fractional derivative are well defined
\textit{a.e.} on $[s,T]$ with the Caputo derivative having a shorter form after $\Ceil{q}$ times
of integration-by-parts (see \cite[Theorem 2.2, p.39]{SamkoKilbasMarichev93} or \cite[Theorem 2.1,
p.92]{KilbasEtc06}):
\begin{align}
		\label{E:caputova}
		\left( \partial_{s+}^q f \right)(t)=\frac{1}{\Gamma(n+1-q)} \int_s^t\:
		\frac{f^{(n+1)}(\tau)}{(t-\tau)^{q-n}}\ud\tau, \qquad n=\Ceil{q}-1.
\end{align}
In particular, when $s=0$, this reduces to the form \eqref{E:caputuder}.

However, in this paper, for $0\le n<q<n+1$, the function $h_s^{(n)}(t)$ is in general not absolutely
continuous over $[s,T]$, or equivalently, $h_s^{(n+1)}(t)$ may not be well defined. When computing
$\left(\partial_{s+}^q h_s \right)(t)$, one cannot blindly apply the expression \eqref{E:caputova}.
We don't presume that $\left(\partial_{s+}^q  h_s \right)(t)$ is well defined until it's shown in
Lemma \ref{L:caputo_h}. We will need the following composition property of Caputo derivative without
assuming $f^{(n)}(t)$ to be absolutely continuous, which shows that the LFD of $f(t)$ of order
$q\in\R_+\setminus\bbZ_+$ is also the $(q-n)$-th order LFD of $f^{(n)}(t)$, where $n=\Floor{q}$.

\begin{lemma}
		\label{L:lfdformvar}
		Suppose $0< n <q < n+1$ and $T>0$.  Let $f^{(n-1)}(t)$ be absolutely
		continuous on $[s,T]$.  If either $\left( \partial_{s+}^q f \right)(t)$  or
		$\left(\partial_{s+}^{q-n} f^{(n)} \right)(t)$ exists on $(s,T)$, then
		$\left(\partial_{s+}^{q} f\right)(t)=\left( \partial_{s+}^{q-n} f^{(n)} \right)(t)$ on $(s,T).$
\end{lemma}
\begin{proof}
		Fix an arbitrary $t\in (s,T)$. Starting from the definition of the Caputo derivative
		\eqref{E:CapDef}, via integration-by-parts, we see that
		\begin{align*}
				\left( \partial_{s+}^{q} f \right)(t)
				= & \frac{1}{\Gamma(n+1-q)} \left(\dv{}{ t}\right)^{n+1}\left\{-\left.\frac{(t-y)^{n+1-q}}{n+1-q}\left[f(y)-\sum_{k=0}^{n} \frac{f^{(k)}(s)}{k !}(y-s)^{k}\right]\right|_{y=s} ^{t}\right. \\
				  & \quad \left.+\int_{s}^{t} \frac{(t-y)^{n+1-q}}{n+1-q} \dv{}{ y}\left[f(y)-\sum_{k=0}^{n} \frac{f^{(k)}(s)}{k !}(y-s)^{k}\right] \ud y\right\}                                         \\
				= & \frac{1}{\Gamma(n+1-q)} \left(\dv{}{ t}\right)^{n+1}\int_{s}^{t} \frac{(t-y)^{n+1-q}}{n+1-q} \left[f'(y)-\sum_{k=1}^{n} \frac{f^{(k)}(s)}{(k-1) !}(y-s)^{k-1}\right] \ud y             \\
				= & \frac{1}{\Gamma(n+1-q)} \left(\dv{}{ t}\right)^{n}\int_{s}^{t} (t-y)^{n-q}\left[f'(y)-\sum_{k=1}^{n} \frac{f^{(k)}(s)}{(k-1) !}(y-s)^{k-1}\right] \ud y.
		\end{align*}
		Repeat the above integration-by-parts $n$ times to see that
		\begin{align}
				\label{E:lfdformvar}
				\left( \partial_{s+}^{q} f \right)(t) =
				\frac{1}{\Gamma(n+1-q)}\dv{}{ t}\int_{s}^{t} (t-y)^{n-q}\left[f^{(n)}(y)-f^{(n)}(s)\right] \ud y,
		\end{align}
		which is nothing but what we need to prove, namely, $\left(\partial_{s+}^{q} f\right)(t)=\left(
		\partial_{s+}^{q-n} f^{(n)} \right)(t)$.
\end{proof}
\begin{remark}
		If $f^{(n)}$ is absolutely continuous on
		$[s,T]$ as is assumed in Theorem 2.1 of \cite{KilbasEtc06}, one can apply one more time this
		integration-by-parts to arrive at \eqref{E:caputova}.  Here, we will postpone the determination of
		the well-posedness of $\left( \partial_{s+}^{q} f^{(n)}\right) (t)$ latter.
\end{remark}
The following \textit{Fractional Taylor theorem} is crucial for the estimation of $h_s(t)$.
\begin{lemma}[Fractional Taylor theorem] \label{L:ftt}
		Let $f$ be a real-valued function. Assume that for some $q>0$ and $s\in\R$, the following two
		conditions hold:
		\begin{enumerate}
				\item $f^{(\Ceil{q}-1)}(t)$ be right-continuous at $s$, namely, $\lim_{t\to s_+}
						f^{(\Ceil{q}-1)}(t) = f^{(\Ceil{q}-1)}(s)$;
				\item $\LFD^qf(s)$ exists;
		\end{enumerate}
		then
		\begin{align}
				\label{E:ftt}
				f(t)-\sum_{k=0}^{\Ceil{q}-1}\frac{f^{(k)}(s)}{k!}(t-s)^k= \frac{\LFD^qf(s)}{\Gamma(q+1)} (t-s)^q+o\left[(t-s)^{q}\right],
				\quad \text{as $t\downarrow s$.}
		\end{align}
\end{lemma}
\begin{remark}
		This lemma is a rephrasing of \cite[Theorems 4 and 5]{BenCre13}, which has been both tailored
		for our application and generalized from $q\in (0,1)$ to all $q>0$. To the best of our
		knowledge, Lemma \ref{L:ftt} also appeared, with stronger or equivalent conditions, in the
		following references: In the case of $n=0$, \cite{KoGaLFD} proved it under a stronger condition
		that $\left( \partial_{s+}^{q} f \right)(t)$ is absolutely continuous; Proposition 2 of
		\cite{CYK10} in fact needs that $( I_{s+}^{1-q} f)(t)$ is absolutely continuous. Note that both
		Proposition 3.1 of \cite{Jum06} and Theorem 1.12 of \cite{YBS16} give different types of
		Fractional Taylor theorem under more stringent conditions on $f(t)$. For completeness and also
		for the readers' convenience, we reproduce below the proof of this lemma with slight more
		details.
\end{remark}

\begin{proof}[Proof of Lemma \ref{L:ftt}]
		When $q\ge 1$ is an integer, \eqref{E:ftt} reduces to the classical Taylor's expansion. In the
		following, we will assume that $q\not\in\mathbb{Z}$.

		Let us first consider the case $0<q<1$. Since $f(t)$ is right continuous at $t=s$, there exists
		some $T>s$ such $f\in C([s,T])\subseteq L^1([s,T])$. Therefore, for all $t\in [s,T]$,
		\begin{align}
				\notag
				f(t)-f(s) = & \dv{}{t} \left( I_{s+}^{1} [ f(\cdot)-f(s) ] \right)(t)
                  = \dv{}{t}\left( I_{s+}^{q}\circ I_{s+}^{1-q} [ f(\cdot)-f(s) ] \right)(t) \\
                  = & \dv{}{t}\frac{1}{\Gamma(q)}\int_s^t \left(I_{s+}^{1-q} [ f(\cdot)-f(s) ]\right )(y)\: (t-y)^{q-1} \ud y,
			 \label{E_:f-f}
		\end{align}
		where the first and the second equalities are due to the fundamental theorem of calculus and
		Lemma 2.3 on p. 73 of \cite{KilbasEtc06}, respectively, both applied to continuous functions.
		Now in order to apply the integration-by-parts formula to the above $\ud y$-integral, we need to
		show that
		\begin{align}
				\label{E:AB}
				\text{the function $y\mapsto \left( I_{s+}^{1-q} \left[ f(\cdot)-f(s)\right ] \right)(y)$ is absolutely continuous on $[s,T]$.}
		\end{align}
		Indeed, by \eqref{E:CapDef},
		\begin{align*}
				\frac{\ud }{\ud y}\left(I_{s+}^{1-q} \left[ f(\cdot)-f(s) \right]\right)(y)
				= (\partial_{s+}^q f)(y),
		\end{align*}
		where the right-hand side is well defined, at least at a small right neighborhood of $s$, thanks
		to the condition that $\LFD^q f(s)$ exists. Hence, by moving $T$ towards $s$ whenever necessary,
		we see that $y\mapsto \left( I_{s+}^{1-q} \left[ f(\cdot)-f(s)\right ] \right)(y)$ is  is
		differential {\it everywhere} on $[s,T]$ with its derivative equal to $\left(D_{s+}^{q} \left[
		f(\cdot)-f(s) \right]\right)(y)$ or, equivalently, $\left(\partial_{s+}^{q} f\right)(y)$ (see
		\eqref{E:CapDef}). Again, the existence of $\LFD^q f(s)$ guarantees that
		\begin{align}
				\label{E:Df-L1}
				(\partial_{s+}^qf)(\cdot) \in L^{\infty}([s,T])\cap L^1([s,T]).
		\end{align}
		This proves \eqref{E:AB}.  Hence, we  can apply the integration-by-parts to see that for all
		$y\in [s,T]$, the right-hand side of \eqref{E_:f-f} is equal to
		\begin{align*}
				= & \dv{}{t}\frac{1}{\Gamma(q)}\bigg[\: -\frac1q (t-y)^q \times \left(I_{s+}^{1-q} [ f(\cdot)-f(s) ]\right)(y) \Big|_{y=s}^{y=t}
					 + \int_s^t  \left( \partial_{s+}^{q} f \right)(y) \times \frac1q (t-y)^{q} \ud y \:\bigg]   \\
				= & \dv{}{t}\frac{1}{\Gamma(q)}\int_s^t \left( \partial_{s+}^{q} f \right)(y)\times \frac1q (t-y)^{q} \ud y,
		\end{align*}
		where the first term in the bracket is equal to zero because $f\in C([s,T])\subseteq
		L^1([s,T])$. Thanks to the boundedness of the integrand (see \eqref{E:Df-L1}), one can switch the $\ud t$ derivative and the $\ud y$
		integral to see that
		\begin{align}
				\label{E_:ftt1}
				f(t)-f(s)= \left( I_{s+}^{q}\circ D_{s+}^{q} \left[ f(\cdot)-f(s) \right] \right)(t),\quad \text{for all $t\in[s,T]$},
		\end{align}
		where the fractional integral is well defined thanks to \eqref{E:Df-L1}. Finally, the definition
		of LFD in \eqref{E:LFD-Def} implies that
		\begin{align}
				\label{E_:ftt2}
				\left( I_{s+}^{q} \circ D_{s+}^{q} [ f(\cdot)-f(s) ] \right)(t)
				= \left( I_{s+}^{q} [ \LFD^qf(\cdot)+o(1) ] \right)(t)
				= \frac{\LFD^q f(s)}{\Gamma(q+1)} (t-s)^q+o\left[(t-s)^{q}\right].
		\end{align}
		Combining \eqref{E_:ftt1} and \eqref{E_:ftt2} proves \eqref{E:ftt} for the case when $q\in (0,1)$.
		\bigskip

		It remains to prove the case when $q>1$ and $q\not\in\mathbb{Z}$. Notice that by Lemma
		\ref{L:lfdformvar}, the existence of $\LFD^q f(s)$ is equivalent to the existence of
		$\LFD^{q-\Floor{q}}f^{(\Floor{q})}(s)$. Now we can apply the case of $0<q<1$ to
		$f^{(\Floor{q})}(t)$ to see that
		\begin{align}\label{lfftt}
				f^{(\Floor{q})}(t)-f^{(\Floor{q})}(s)
				=&\frac{\LFD^{q-\Floor{q}}f^{(\Floor{q})}(s)}{\Gamma(q-\Floor{q}+1)}(t-s)^{q-\Floor{q}}+o\left((t-s)^{q-\Floor{q}}\right).
		\end{align}
		Finally, the case is proved by applying the $I_{s+}^{\Floor{q}}$ on both sides of \eqref{lfftt}.
		This completes the proof of Lemma \ref{L:ftt}.
\end{proof}

\subsection{Local fractional derivative of \texorpdfstring{$h_s^{(n)}(t)$}{} }
\label{S:lfdhn}

By Lemma \ref{L:ftt} and Lemma \ref{L:lfdformvar}, to prove Proposition \ref{P:Holder_h}  we only need
to show  $\LFD^{q-n}h_s^{(n)}(t)=0$ for $0\leq n<q<n-1$. We will apply Fubini's Theorem to show
the interchangeability of local fractional operator $\LFD^q $ and the integral operator
$\int^{s}_{0}\int_{\R^d}$ in $h_s^{(n)}(t)$ in the following two lemmas.

\begin{lemma}
		\label{L:FracInt}
		Assuming \eqref{E:nq}, for all $T\in (s,\infty)$, it holds that
		\begin{align}
				\label{E:FracInt_con}
				t\mapsto \left( I_{s+}^{n+1-q} \left[h_s^{(n)}(\cdot)-h_s^{(n)}(s)\right] \right)(t) \in C([s,T])
		\end{align}
		and for all $t\in[s,T]$,
		\begin{multline}\label{E:FraInt_h}
			\left( I_{s+}^{n+1-q}  [h_s^{(n)}(\cdot)- h_s^{(n)}(s)] \right)(t) \\
			= \int^{s}_{0}  \ud r\int_{\R^d} \ud \xi\left( I_{s+}^{n+1-q} \left[\mathfrak{E}(\cdot,\beta+\gamma-n,r,\xi) - \mathfrak{E}(s,\beta+\gamma-n,r,\xi)\right] \right)(t).
		\end{multline}
\end{lemma}
\begin{proof}
		The statement of \eqref{E:FracInt_con} is due to the fact that $h_s^{(n)}(t)\in C([s,T])$; see
		Lemma \ref{L:intdevh(s)}.  As for \eqref{E:FraInt_h}, there are two fractional integrals. The
		second one is trivially true since since $h_s^{(n)}(s)$ does not depend on $t$. For the first
		fractional integral, we can switch the order of integrations because, by \eqref{E:intbdd_h^(n)},
		\begin{align*}
				\left( I_{s+}^{n+1-q} \int^{s}_{0}\ud r \int_{\R^d} \ud \xi \left| \mathfrak{E}(\cdot,\beta+\gamma-n,r,\xi)\right|\right) (t)
				< \left( I_{s+}^{n+1-q} \:  C_T \right) (t) < C_T'< \infty.
		\end{align*}
		This proves the lemma.
\end{proof}

\begin{lemma}
		\label{L:caputo_h}
		Assuming \eqref{E:nq}, for all $T\in (s,\infty)$, it holds that $ \left( \partial_{s+}^{q-n} h_s^{(n)} \right)(t)\in C([s,T])$ and
		\begin{align}
				\label{E:caputo_h}
				\left( \partial_{s+}^{q-n} h_s^{(n)} \right)(t)=\int^{s}_{0} \ud r \int_{\R^d} \ud \xi\;\;
				\left(\partial_{s+}^{q-n} \mathfrak{E}(\cdot,\beta+\gamma-n,r,\xi)\right)(t)
				\quad \text{for all $t\in[s,T]$.}
		\end{align}
		Consequently,
		\begin{align}
				\label{E:lfd_h}
				\LFD^q h_s(s) = \LFD^{q-n} h_s^{(n)}(s) = 0,\qquad\text{for all $s>0$.}
		\end{align}
\end{lemma}
\begin{proof}
		We will prove this lemma in the following three steps: \bigskip

		{\bf\noindent Step 1.~} We will first verify the condition of Fubini's theorem is satisfied,
		so that \eqref{E:caputo_h} is valid \textit{a.e.} on $[s, T]$. By the definition of the Caputo
		derivative \eqref{E:CapDef} and Lemma \ref{L:FracInt},
		\begin{align*}
				\left(\partial_{s+}^{q-n}h_s^{(n)}\right)(t)
				& = \frac{\ud}{\ud t} \left( I_{s+}^{n+1-q}\left[ h_s^{(n)}(\cdot) - h_s^{(n)}(s) \right ] \right) (t)\\
				& = \frac{\ud}{\ud t} \int^{s}_{0}  \ud r\int_{\R^d} \ud \xi\left( I_{s+}^{n+1-q} \left[\mathfrak{E}(\cdot,\beta+\gamma-n,r,\xi) - \mathfrak{E}(s,\beta+\gamma-n,r,\xi)\right] \right)(t).
		\end{align*}
		Noticing that
		\begin{align*}
				\frac{\ud}{\ud t} & \left( I_{s+}^{n+1-q} \left[\mathfrak{E}(\cdot,\beta+\gamma-n,r,\xi) - \mathfrak{E}(s,\beta+\gamma-n,r,\xi)\right] \right)(t)
				=  \left( \partial_{s+}^{q-n} \mathfrak{E}(\cdot,\beta+\gamma-n,r,\xi) \right)(t),
		\end{align*}
		\eqref{E:caputo_h} is valid \textit{a.e.} on $[s, T]$ once one can justify that the derivative
		of $t$ can be passed through the double integrals.  To switch the derivative and the integrals,
		one can carry out some similar, but much more involved, dominated convergence arguments in two
		steps as the proof of Lemma \ref{L:intdevh(s)}. Here we will use a slightly different, but more
		convenient, criterion --- Theorem 4 of \cite{Talvila01iff}. Indeed, notice that, by \eqref{E:caputova},
		\begin{align*}
				\left( \partial_{s+}^{q-n} \mathfrak{E}(\cdot,\beta+\gamma-n,r,\xi) \right)(t)
				= \frac{1}{\Gamma(n+1-q)} \int_s^t (t-y)^{n-q}\mathfrak{E}(y,\beta+\gamma-n-1,r,\xi) \ud y,
		\end{align*}
		In light of Theorem 4 of \cite{Talvila01iff}, it suffices to show that
		\begin{align*}
				J:= \int^{s}_{0}\ud r\;\int_{\R^d} \ud \xi\; \int_s^t \ud y \: (t-y)^{n-q} \left|\mathfrak{E}(y,\beta+\gamma-n-1,r,\xi)\right|<\infty.
		\end{align*}
		An application of Lemma \ref{L:Y^2} shows that
		\begin{align}
				 \label{E:Jbdd}
				 J \leq & C_{*}'\int^{s}_{0}\ud r\int^{t}_{s}\ud y\: (y-r)^{\frac{\rho-3}{2}-n}(t-y)^{n-q}(s-r)^{\frac{\rho-1}{2}},
		\end{align}
		where $C_{*}':=\sqrt{C_{\gamma-n-1}C_{\gamma}}$. When $t<2s$, we have
		\begin{align*}
				J = & C_{*}' \left( \int_0^{2s-t}\ud r +\int_{2s-t}^s \ud r \right) \int^{t}_{s}\ud y\:
				(y-r)^{\frac{\rho-3}{2}-n}(t-y)^{n-q}(s-r)^{\frac{\rho-1}{2}}
				=: C_{*}'\left(J_1+J_2\right).
		\end{align*}
		We first consider $J_1$. Because $\frac{t+r}{2}\le s$, we have that the following:
		\begin{align*}
			 J_1 \leq & \int^{2s-t}_{0}\ud r \int^{t}_{\frac{t+r}{2}}\ud y\: (y-r)^{\frac{\rho-3}{2}-n}(t-y)^{n-q}(s-r)^{\frac{\rho-1}{2}}                      \\
			     \leq & \int^{2s-t}_{0}\ud r \int^{t}_{\frac{t+r}{2}}\ud y\: \left(\frac{t-r}{2}\right)^{\frac{\rho-3}{2}-n}(t-y)^{n-q}(s-r)^{\frac{\rho-1}{2}} \\
			     =    & \frac{1}{n+1-q} \int^{2s-t}_{0}\ud r \: \left(\frac{t-r}{2}\right)^{\frac{\rho-1}{2}-q}(s-r)^{\frac{\rho-1}{2}}                         \\
			     \leq & \frac{1}{n+1-q} \int^{2s-t}_{0}\ud r \: \left(\frac{s-r}{2}\right)^{\frac{\rho-1}{2}-q}(s-r)^{\frac{\rho-1}{2}}                         \\
			     \leq & \frac{1}{n+1-q} \int^{s}_{0}\ud r \: \left(\frac{s-r}{2}\right)^{\frac{\rho-1}{2}-q}(s-r)^{\frac{\rho-1}{2}}                            \\
					 =    & \frac{1}{n+1-q} 2^{q-\frac{\rho-1}{2}} \frac{1}{\rho-q} s^{\rho-q},
		\end{align*}
		where the second inequality is due to $\frac{t-r}{2}\le y-r$ and the third inequality is thanks
		to $\frac{\rho-1}{2}-q<0$, and recall from \eqref{E:nq} that $\rho-q\in(0,1)$.
		As for $J_2$, since $\frac{t+r}{2}\ge s$, we see that
		\begin{align*}
			 J_2 = & \int_{2s-t}^{s}\ud r \left(\int_{s}^{\frac{t+r}{2}}\ud  y + \int_{\frac{t+r}{2}}^t \ud y\right)(y-r)^{\frac{\rho-3}{2}-n}(t-y)^{n-q}(s-r)^{\frac{\rho-1}{2}}  =: (J_{21}+J_{22}).
		\end{align*}
		Estimation for $J_{22}$ is similar to $J_1$:
		\begin{align*}
				J_{22} \leq & \frac{1}{n+1-q} \int_{2s-t}^{s}\ud r \: \left(\frac{s-r}{2}\right)^{\frac{\rho-1}{2}-q}(s-r)^{\frac{\rho-1}{2}}\\
				\leq & \frac{1}{n+1-q} \int_{0}^{s}\ud r \: \left(\frac{s-r}{2}\right)^{\frac{\rho-1}{2}-q}(s-r)^{\frac{\rho-1}{2}}\\
					 = & \frac{1}{n+1-q} 2^{q-\frac{\rho-1}{2}} \frac{1}{\rho-q} s^{\rho-q}.
		\end{align*}
		Finally, for $J_{21}$, we have that
		\begin{align*}
				J_{21} = & \int^{s}_{2s-t}\ud r\int_{s}^{\frac{t+r}{2}}\ud y\;\;(y-r)^{\frac{\rho-3}{2}-n}(t-y)^{n-q}(s-r)^{\frac{\rho-1}{2}} \\
				 \leq & \int^{s}_{2s-t}\ud r\left(\frac{t-r}2\right)^{n-q} (s-r)^{\frac{\rho-1}{2}}\int_{s}^{\frac{t+r}{2}}(y-r)^{\frac{\rho-3}{2}-n}\ud y  \\
				 = & \frac{1}{{n-\frac{\rho-1}{2}}}\int^{s}_{2s-t}\ud r \left(\frac{t-r}2\right)^{n-q}(s-r)^{\frac{\rho-1}{2}} \left[(s-r)^{\frac{\rho-1}{2}-n}- \left(\frac{t-r}2\right)^{\frac{\rho-1}{2}-n}\right].
		\end{align*}
		where we have used the fact that $n-(\rho-1)/2>0$.  Because $n-q<0$ and $(\rho-1)/2<n$, we can
		replace $ (t-r)/2$ by $(s-r)/2$ to see that
		\begin{align*}
		    J_{21} \leq C^\dagger\int^{s}_{0} (s-r)^{\rho-1-q} \ud r
						    	= \frac{C^\dagger}{\rho-q}s^{\rho-q},
									\quad \text{with $C^\dagger:=\frac{2}{2n-\rho+1}\left[2^{q-n}-2^{q-\frac{\rho-1}{2}}\right].$}
		\end{align*}
		Combining all these terms proves $J<\infty$. When $t>2s$, we have $\frac{t+r}{2}\ge s$ for all
		$r\in(0,s)$. Then $J$ is dominated as $J_{2}$ is in the case of $t<2s$.  Therefore, it is legal
		to switch the derivative and the double integrals to conclude that \eqref{E:caputo_h} holds {\it
		a.e.} on $t\in[s,T]$. \bigskip

		{\bf\noindent Step 2.~} Now we will show $\left(\partial_{s+}^{q-n} h_s^{ (n)}\right)(t)$ is
		continuous on $[s, T]$ by showing that
		\begin{align}
			  \label{E_:limInt}
				\lim_{t\rightarrow t_0}\left(\partial_{s+}^{q-n} h_s^{ (n)}\right)(t)
				=\left(\partial_{s+}^{q-n} h_s^{ (n)}\right)(t_0)\qquad\text{for all $t_0\in[s,T]$},
		\end{align}
		which then implies that \eqref{E:caputo_h} is valid for all $t\in [s, T]$. Notice that
		Step 1 proves that for {\it a.e.} $t\in [s,T]$,
		\begin{align}
				\label{E:caputo_int}
				\left(\partial_{s+}^{q-n} h_s^{ (n)}\right)(t) =
				C \int_{\R}\ud r\int_{\R}\ud y\int _{\R^d}\ud \xi&\;\;
				\One_{[0,s]}(r) \One_{[s,t]}(y)\mathfrak{E}(y,\beta+\gamma-n-1,r,\xi)(t-y)^{n-q}
		\end{align}
		with $C=1/\Gamma(n+1-q)$, where one can choose whatever order of integrations thanks to Fubini's
		theorem. Hence, the proof of \eqref{E_:limInt} amounts to passing the limit through three
		integrals.

		First notice that one can switch the order of the limit $t$ with the $\ud r$-integral because
		from the proof in Step 1, the integrand for the $\ud r$-integral is dominated by a function of
		$(s-r)$ that does not depend on $t$ and is integrable with respect to $\ud r$.

		It remains to show that one can further pass the limit into the rest two integrals, namely,
		\begin{align}
				\label{E_:limInt2}
				\begin{aligned}
						\lim_{t\rightarrow t_0} & \int_{\R} \ud y\int _{\R^d}\ud \xi \,\,\One_{[s,t]}(y)\mathfrak{E}(y,\beta+\gamma-n-1,r,\xi)(t-y)^{n-q} \\
				                      = & \int_{\R}\ud y\int _{\R^d}\ud \xi \,\,\One_{[s,t_0]}(y)\mathfrak{E}(y,\beta+\gamma-n-1,r,\xi)(t_0-y)^{n-q}.
				\end{aligned}
		\end{align}
		In the following, let $C$ be a general constant may vary from line to line.  For $t\in[s,T]$,
		because $n-q\in (-1,0)$ and
		\begin{align*}
			\int_s^t (t-y)^{n-q}\ud y = \frac{1}{n+1-q}(t-s)^{n+1-q} \le \frac{1}{n+1-q}(T-s)^{n+1-q} = \int_s^T (T-y)^{n-q}\ud y,
		\end{align*}
		together with \eqref{E:AsypMTL}, we see that the integrand of the double integral can be bounded
		as
		\begin{multline*}
				\left|\One_{[s,t]}(y)\mathfrak{E}(y,\beta+\gamma-n-1,r,\xi)(t-y)^{n-q}\right| \\
				\leq C(s-r)^{2(\beta+\gamma)-n-3} (T-y)^{n-q}\One_{[s,T]}(y)
				\left(1+2^{-1} \nu\left(s-r\right)^{\beta }|\xi|^{\alpha}\right)^{-2} =: I(y,\xi).
		\end{multline*}
		Notice that this upper bound $I(y,\xi)$ does not depend on $t$. Moreover,
		\begin{align*}
				\int_\R \ud y \int_{\R^d}\ud\xi \: \left|I(y,\xi)\right|
				\leq &C(s-r)^{2(\beta+\gamma)-n-3} \int_s^T\ud y\: (T-y)^{n-q}\int_{\R^d} \ud \xi\,
				\left(1+2^{-1} \nu\left(s-r\right)^{\beta }|\xi|^{\alpha}\right)^{-2} \\
				=    &C(s-r)^{2(\beta+\gamma)-n-3} (T-y)^{n+1-q}\int_{\R^d} \ud \xi\,
				\left(1+2^{-1} \nu\left(s-r\right)^{\beta }|\xi|^{\alpha}\right)^{-2} \\
				<    & \infty.
		\end{align*}
		Then, an application of the dominated convergence proves \eqref{E_:limInt2}.  This proves
		\eqref{E_:limInt}. Therefore, \eqref{E:caputo_h} holds for all $t\in[s,T]$. \bigskip

		{\bf\noindent Step 3.~} Finally, for \eqref{E:lfd_h}, the first equality holds trivially in case
		of $n=0$, or by Lemma \ref{L:lfdformvar} otherwise. As for the second equality, by
		\eqref{E:LFD-Def}, \eqref{E:CapDef} and Step 2, we can push the limit of $t$ inside the triple
		integrals of \eqref{E:caputo_int} to conclude that
		\begin{align*}
				\LFD^{q-n}h_s^{(n)}(s) = \lim_{t\rightarrow s_+} \left( \partial_{s+}^{q-n} h_s^{ (n)}\right)(t) =0.
		\end{align*}
		This completes the whole proof of Lemma \ref{L:caputo_h}.
\end{proof}

\section{Proof of proposition \ref{mP:G-SD}}
\label{S:Prop}

Now we are ready to prove Proposition \ref{mP:G-SD}.

\begin{proof}[Proof of proposition \ref{mP:G-SD}]
Denote the left-hand side of \eqref{mE:G-t2} by $I$ and then
\begin{align*}
		I=  C_{\gamma}\int_0^s \ud r \left[(t-r)^{2(\beta+\gamma-1)-d\beta/\alpha}+(s-r)^{2(\beta+\gamma-1)-d\beta/\alpha}\right]- 2 h_s(t).
\end{align*}

When $\rho\leq1$, by Proposition \ref{P:Holder_h}, we have $h_s(t)=h_s(s)+o\left[(t-s)^{q}\right]$ and
so there is a constant $K$, such that $|h_s(t)-h_s(s)|<K(t-s)^q$. Therefore,
\begin{align*}
		I & \leq \frac{ C_{\gamma}}{\rho}\left[t^{\rho}-(t-s)^{\rho} + s^{\rho}\right] -2\left[h_s(s)-K(t-s)^q\right] \\
			& = \frac{ C_{\gamma}}{\rho}\left[t^{\rho}-(t-s)^{\rho} + s^{\rho} -2s^\rho\right]+2 K(t-s)^q      \\
			& = O\left[(t-s)^{\rho}\right]+2 K(t-s)^q,
\end{align*}
with $2 K(t-s)^q$ being the dominating term. \bigskip

Similarly, when $1<\rho\leq 2$, by Proposition \ref{P:Holder_h} and Lemma \ref{L:intdevh(s)},
\begin{align*}
		I\leq & \frac{ C_{\gamma}}{\rho}\left[t^{\rho}-(t-s)^{\rho}+s^{\rho}\right]-2 \left[h_s(s)+\frac12C_{\gamma}s^{\rho-1}(t-s)- K(t-s)^q\right]     \\
		    = &  C_{\gamma}\left[\frac{1}{\rho}\left(t^{\rho}-s^{\rho}\right)-s^{\rho-1}(t-s)\right] - \frac{ C_{\gamma}}{\rho}(t-s)^{\rho}+2*K(t-s)^q \\
				= & \frac{ C_{\gamma}(\rho-1)s^{\rho-2}}{2}(t-s)^2 + O\left[(t-s)^3\right] + O\left[(t-s)^{\rho}\right]+2 K(t-s)^q,
\end{align*}
with $2 K(t-s)^q$ being the dominating term. \bigskip

Finally, when $\rho>2$,
\begin{align*}
		I\leq & \frac{ C_{\gamma}}{\rho}\left[t^{\rho}-(t-s)^{\rho}+s^{\rho}\right]-2 \left[h_s(s)+\frac12C_{\gamma}s^{\rho-1}(t-s)+\frac12h_s''(s)(t-s)^2- K(t-s)^q\right] \\
		    = & C_{\gamma}\left[\frac{1}{\rho}\left(t^{\rho}-s^{\rho}\right)-s^{\rho-1}(t-s)\right] - \frac{ C_{\gamma}}{\rho}(t-s)^{\rho}-h_s''(s)(t-s)^2+2K(t-s)^q        \\
				= & \frac{ C_{\gamma}(\rho-1)s^{\rho-2}}{2}(t-s)^2 + O\left[(t-s)^3\right] + O\left[(t-s)^{\rho}\right] -h_s''(s)(t-s)^2+2K(t-s)^q                              \\
				= & \left(\frac{C_{\gamma}(\rho-1)}{2}+\frac{C_{\gamma-1}}{\rho-2}-C_{\gamma,\gamma-1}\right)s^{\rho-2}(t-s)^2+O\left[(t-s)^3\right]+O\left[(t-s)^{\rho}\right] +2K(t-s)^q,
\end{align*}
where the second order term is the dominating term.
This completes the proof of Proposition \ref{mP:G-SD}.
\end{proof}
\addcontentsline{toc}{section}{Bibliography}
\bibliographystyle{abbrv}

\end{document}